\documentclass[10pt]{article}
\usepackage{graphicx,amsmath,amsfonts,amsthm,amssymb,verbatim,stmaryrd}
\newtheorem{theorem}{Theorem}
\newtheorem{lemma}[theorem]{Lemma}
\newtheorem{prop}[theorem]{Proposition}

\newtheorem{cor}[theorem]{Corollary}

\theoremstyle{remark}
\newtheorem*{remark}{Remark}

\newcounter{qcounter}

\newcommand{\R}{\mathbb R}
\newcommand{\C}{\mathbb C}
\newcommand{\Z}{\mathbb Z}
\newcommand{\Q}{\mathbb Q}
\newcommand{\F}{\mathbb F}

\newcommand{\p}{\mathfrak p}
\newcommand{\Pp}{\mathfrak P}

\newcommand{\g}{\mathfrak g}
\newcommand{\Aa}{\mathfrak a}

\newcommand{\SL}{\text{SL} }
\newcommand{\PSL}{\text{PSL} }
\newcommand{\OO}{\mathcal O}

\newcommand{\im}{\text{im }}

\newcommand{\Sym}{\text{Sym}}
\newcommand{\Hom}{\text{Hom}}
\newcommand{ \Galg}{ \F_p \llbracket G \rrbracket }
\newcommand{ \Halg}{ \F_p \llbracket H \rrbracket }

\begin{document}

\title{Bounds for Multiplicities of Automorphic Forms of Cohomological Type on $GL_2$}
\author{Simon Marshall}

\maketitle

\begin{abstract}
This paper contains various results on the cohomology of arithmetic lattices arising from quaternion algebras over a number field with at least one complex place.  Our main result is a power saving for the dimension of the space of cohomological automorphic forms when the associated locally symmetric space is a hyperbolic 3-manifold.  We also prove that cuspidal cohomological forms on these groups must have the same weights as a form base changed from a totally real subfield, and simplify a criterion of Boston and Ellenberg for identifying exhaustive towers of hyperbolic 3-manifolds with trivial rational cohomology.
\end{abstract}

\section{Introduction}
\label{coh1}

Given a semisimple Lie group $G$, a sublattice $\Gamma$, and a coefficient system $\rho$, it is a natural problem to estimate the dimension of the cohomology spaces $H^i( \Gamma, \rho )$, or equivalently the multiplicity $m(\pi, \Gamma)$ with which automorphic forms of cohomological type appear in $L^2_{\text{cusp}} ( \Gamma \backslash G )$.  If $G$ has a compact Cartan subgroup and $\pi$ is in the discrete series it is possible to compute the numbers $m(\pi, \Gamma)$ using Euler-Poincar\'e characteristics \cite{Se} or the trace formula \cite{A}, and if the highest weight of $\rho$ is regular one may in fact give an explicit formula for $\dim H^i( \Gamma, \rho )$ \cite{A}.

If $\pi$ is not discrete series then no explicit formula for $m(\pi, \Gamma)$ is known, and it then becomes natural to ask what the asymptotic behaviour of these multiplicities is, either as $\Gamma$ varies over sublattices of a fixed lattice (the level aspect) or $\rho$ varies (the weight aspect).  To describe existing results in this direction, let us introduce some notation.  Suppose that $G = {\bf G}(\R \otimes_\Q F)$ for some connected semisimple algebraic group $\bf{G}$ over a number field $F$, and fix an embedding ${\bf G} \rightarrow GL_n$ for some $n$.  For any ideal $\Aa$ of $\OO_F$ let $G(\Aa)$ denote the intersection of $G$ with the principal congruence subgroup of level $\Aa$ in $GL_n( \OO_F)$, and suppose that $\Gamma$ and $G( \OO_F )$ are commensurable so that the groups $\Gamma( \Aa ) = \Gamma \cap G( \Aa )$ are of finite index in $\Gamma$.  If we let $V(\Aa) = | \Gamma : \Gamma( \Aa ) |$, the trivial bound for $m(\pi, \Gamma(\Aa) )$ may be stated as $m(\pi, \Gamma(\Aa) ) \ll V(\Aa)$.  The basic result on limit multiplicities is that for any $\pi$ which is not discrete series this may be strengthened to

\begin{equation*}
m( \pi, \Gamma(\Aa) ) = o( V( \Aa ) ),
\end{equation*}

which was derived from the trace formula by de George - Wallach \cite{dGW}, Luck and Savin \cite{Sa}, while for spaces of real rank 1 Peter Sarnak informs me that this can be quantified in the form of

\begin{equation*}
m( \pi, \Gamma(\Aa) ) \ll \frac{ V( \Aa ) }{ \ln V( \Aa ) }.
\end{equation*}

It is interesting to find cases where we may improve over the trivial bound by a power, i.e. give results of the form

\begin{equation*}
m( \pi, \Gamma(\Aa) ) \ll V( \Aa )^{1-\delta}
\end{equation*}

for some $\delta > 0$.  Such bounds follow easily from the trace formula in the case when $\pi$ is nontempered, although the subject of exactly what $\delta$ may be achieved is interesting, and the subject of a conjecture of Sarnak and Xue \cite{SX}.  Our main result is a conditional power saving for families of tempered cohomological forms on $SL(2,\C)$ as the weight varies.

The first example of a power saving for the multiplicity of a tempered automorphic representation was given by Duke \cite{D} for holomorphic cusp forms of weight 1, which by Deligne-Serre \cite{DS} correspond to odd two dimensional complex Galois representations.  Duke considers the dimension of $S_1(q)$, the space of holomorphic cusp forms of weight 1, level $q$ and central character the Legendre symbol $( \cdot / q)$, and improves the trivial bound $\dim S_1(q) \ll q$ to

\begin{equation*}
\dim S_1(q) \ll q^{11/12} \ln^4 q.
\end{equation*}

This was subsequently improved by Michel and Venkatesh to $\ll q^{6/7 + \epsilon}$ \cite{MV}.  A different class of examples was given by Calegari and Emerton in \cite{CE1} for so-called $p$-adic congruence towers, i.e. those obtained by imposing congruence conditions at powers of a fixed prime ideal.  With notation as above, let $\p$ be a prime ideal of $\OO_F$ and $\pi \in \widehat{G}$ be of cohomological type.  Suppose that $G$ does not admit discrete series, or that if it does, that $\pi$ contributes to cohomology outside dimension $1/2 \dim( G / K)$ where $K < G$ is the maximal compact.  Calegari and Emerton then prove the bound

\begin{equation*}
m( \pi, \Gamma( \p^k ) ) \ll V(\p^k)^{1 - 1/\dim(G) } \quad \text{as} \; k \rightarrow \infty.
\end{equation*}

We shall loosely describe our results before stating them fully in the remainder of the section.  To give an example of our power saving theorem, let $\Gamma \subset SL(2,\C)$ be a congruence lattice and let $E_d = \Sym^d \otimes \overline{\Sym}^d$ be one of the irreducible self-dual representations of $SL(2,\C)$.  The trivial bound for $\dim H^1(\Gamma, E_d)$ is $\ll d^2$, and we are able to strengthen this to $\ll d^{2-\delta}$ for $\delta > 0$ under certain assumptions which are expected to be true in every case, and can be verified in any example.  This strengthens a result of Finis, Grunewald and Tirao \cite{FGT} who prove that

\begin{equation}
\label{log}
\dim H^1(\Gamma, E_d) \ll d^2 / \ln(d)
\end{equation}

for arbitrary lattices $\Gamma$ using the trace formula.  Our theorem is a result of studying $p$-adic congruence towers in the way developed by Calegari and Emerton \cite{CE1}, \cite{CE2}, and in doing so we have also been able to simplify a criterion of Boston and Ellenberg \cite{BE} which allows one to find exhaustive towers of coverings of hyperbolic three manifolds by rational homology spheres.  Roughly speaking, we show that if $\Gamma \subset SL(2,\C)$ is congruence and $\p$ a degree 1 prime of its trace field, then if $\dim H^1( \Gamma(\p), \F_p) = 3$ then $\dim H^1( \Gamma(\p^k), \F_p) = 3$ and $H^1( \Gamma(\p^k), \Q) = 0$ for all $k$.  Our third result states that if $\pi$ is a cuspidal cohomological form on a quaternion algebra over a number field $F$, the weights of $\pi$ must be the same as those of a form base changed from a totally real subfield of $F$.

I would like to thank my advisor Peter Sarnak for all of his guidance and encouragement while doing this work.  I would also like to thank Tykal Venkataramana for pointing out an error in an earlier approach to this problem, and Frank Calegari and Matthew Emerton for explaining their work to me and many interesting discussions.

\subsection{Power Savings in the Weight Aspect}
\label{coh11}

Let $r_1$ and $r_2$ be positive integers, $\F = \R^{r_1} \times \C^{r_2}$ and $SL(2,\F)$ the corresponding product of copies of $SL(2,\R)$ and $SL(2,\C)$.  We shall use the usual terms such as congruence and arithmetic to describe lattices $\Gamma \subset SL(2,\F)$, and will use the term `stably arithmetic' to describe arithmetic lattices which realise their invariant trace field and quaternion algebra, and so have an embedding into the multiplicative group of this algebra.  All our lattices are assumed to be irreducible.  If $\Gamma$ is an arithmetic lattice with invariant algebra and trace field $A/F$, we let $F$ have $r_0$ real places at which $A$ is ramified, $r_1$ at which it is split and $r_2$ complex places, and let $\sigma_1, \ldots, \sigma_n$ be the collection of Archimedean embeddings (so that each complex place is counted twice).  If $\Gamma$ is stable, each embedding $\sigma_i$ of $F$ gives a map $\sigma_i: \Gamma \rightarrow G_i$ for $G_i = SU(2), SL(2,\R)$ or $SL(2,\C)$, and their product is a map

\begin{equation*}
\Gamma \longrightarrow G = SU(2)^{r_0} \times SL(2,\R)^{r_1} \times SL(2,\C)^{r_2}
\end{equation*}

whose image is a lattice.  If $\Sym^d \circ \sigma_i$ denotes the composition of $\sigma_i: \Gamma \rightarrow G_i$ with the $d$th symmetric power representation of $G_i$ we may define the representations $\rho_d$ for $d = (d_1, \ldots, d_n)$, which are the tensor products

\begin{equation*}
\rho_d = \bigotimes_{i=1}^n \Sym^{d_i} \circ \sigma_i.
\end{equation*}

We shall primarily be interested in the cohomology groups $H^i(\Gamma, \rho_d)$, which we have defined for stably arithmetic $\Gamma$ but can be extended to arbitrary $\Gamma$ for certain weights.  These groups may be described in terms of automorphic forms on $\Gamma \backslash G$ in a way which we shall briefly recall.  Firstly, if $\Gamma$ is not uniform let $\{ P_j \}$ be the set of $\Gamma$-conjugacy classes of parabolic subgroups, and define $H^i_{\text{cusp}}( \Gamma, \rho_d)$ to be the common kernel of the restriction maps to $H^i( P_j, \rho_d)$.  Otherwise, $H^i_{\text{cusp}}( \Gamma, \rho_d)$ equals $H^i( \Gamma, \rho_d)$.  It is a theorem of Borel and Wallach \cite{BW} that $H^i_{\text{cusp}}(\Gamma, \rho_d) = 0$ unless $d_i = d_j$ for $\sigma_i$ and $\sigma_j$ complex conjugate embeddings, which follows from their calculation of the $(\g, K)$ cohomology of admissible irreducible representations of $SL(2,\C)$.

$H^i_{\text{cusp}}( \Gamma, \rho_d)$ is determined by a certain space of cuspidal automorphic forms on $\Gamma \backslash G$.  It follows from the irreducibility of $\Gamma$ that outside the trivial degrees $i = 0$ and $2r_1 + 3r_2$ (the dimension of the locally symmetric space associated to $\Gamma$), $H^i_{\text{cusp}}( \Gamma, \rho_d)$ may only be nonzero for $r_1 + r_2 \le i \le r_1 + 2r_2$.  For $i = r_1 + r_2$, $H^i_{\text{cusp}}( \Gamma, \rho_d)$ is isomorphic to the space of cusp forms on $\Gamma \backslash G$ which are the $d_i + 1$ dimensional representation of $SU(2)$ at ramified real embeddings $\sigma_i$, discrete series of weight $d_i+2$ for $\sigma_i$ split real, and cohomological type on $SL(2,\C)$ of weight $2d_i+2$ for $\sigma_i$ complex.  For other $i$, it is some number of copies of this space.  Our first result is a strengthening of Borel and Wallach's theorem for stably arithmetic $\Gamma$:

\begin{prop}
\label{vanishing}
Let $\Gamma \subset SL(2,\F)$ be a stably arithmetic lattice with trace field $F$.  If $H^i_{\text{cusp}}(\Gamma, \rho_d) \neq 0$ for some $i$, then we must have $d_i = d_j$ for $\sigma_i$ and $\sigma_j$ lying over the same embedding of the maximal totally real subfield of $F$.
\end{prop}

In the language of automorphic forms, this may be rephrased as saying that cuspidal forms of cohomological type may only occur with weights that look as if they were base changed from a totally real subfield.  The proof works by using the Galois group to transfer Borel and Wallach's condition between different pairs of embeddings.  Note that we may define the representations $E_d$ for any lattice $\Gamma \subset SL(2,\C)$, and so for arithmetic $\Gamma$ proposition \ref{vanishing} may also be applied by passing to a strictly arithmetic subgroup of finite index.  Consequently, $H^1_{\text{cusp}} (\Gamma, E_d) = 0$ for $d \neq 0$ unless the invariant trace field of $\Gamma$ has a totally real subfield of index 2.  

Our main result applies to congruence lattices in $SL(2,\C)$, and contains as a special case a conditional strengthening of (\ref{log}) to a power saving.

\begin{theorem}
\label{weight}
Let $\Gamma$ be a congruence lattice in $SL_2(\C)$ with trace field $F$ and quaternion algebra $A$.  Let $p$ be a prime of $\Q$ which is totally split in $F$, with $\p_1, \ldots, \p_n$ the primes above it, and such that $A$ is split at all places above $p$.  Suppose that

\begin{equation}
\label{condition}
\dim_{\F_p} H^1( \Gamma(p \p_i^k), \F_p ) = o( p^{3k} )
\end{equation}

for all $\p_i$.  Then there exists a $\delta > 0$ such that

\begin{eqnarray*}
\dim H^1(\Gamma, \rho_d) \ll \left( \prod_{i=1}^n d_i \right)^{1-\delta}.
\end{eqnarray*}

\end{theorem}

Note that the trivial bound for the dimension of this cohomology group is $\dim \rho_d \sim \prod d_i$.  It has been conjectured by Calegari and Emerton \cite{CE2} that the hypothesis (\ref{condition}) is satisfied for \textit{every} prime, and in section \ref{coh12} we show that it may be checked in any example in which it is true.  We do this for the Bianchi group $SL(2,\OO_{-2})$ and the prime 3 in section \ref{coh41}, producing a bound of $\delta \ge 1/8$.  We believe that this result should also hold for congruence lattices in a general $SL(2,\F)$, but at present the manipulations of cohomology groups in the proof break down when the cohomological degree is greater than 1.  As a result we can only obtain information about $H^1$, which is trivial in higher rank cases.  However, for congruence lattices in $SL(2,\F)$ with $r_1$ even and $r_2 = 1$ we may use the Jacquet-Langlands correspondence to reduce to the case of $SL(2,\C)$, and so obtain a power saving for multiplicities in these groups.

Proposition \ref{vanishing} and theorem \ref{weight} combine with work of Rajan \cite{Ra} and Finis, Grunewald and Tirao \cite{FGT} to suggest a picture of the family of cohomological forms on quaternion algebras over a number field.  If $\Gamma$ is congruence and $\Delta$ is a suitable subgroup, Rajan uses solvable base change to establish the existence of nontrivial elements of $H^i_{\text{cusp}}(\Delta, \rho_d)$ for exactly the weights $d$ we allow.  More precisely, he shows that if $\Gamma \subset SL(2,\F)$ is a congruence lattice with trace field $L$ which lies in a solvable extension of its maximal totally real subfield, and $d$ is an admissible weight, then there is $\Delta \subset \Gamma$ of finite index such that $H^i_{\text{cusp}} ( \Delta, \rho_d) \neq 0$.  Finis, Grunewald and Tirao carry out base change explicitly for the Bianchi groups, proving the lower bound $\dim H^1(\Gamma, E_d) \gg d$.  They also observe the striking behaviour that for $d$ sufficiently large the only contribution to $H^1_{\text{cusp}}$ seemed to be coming from base change and CM forms (a subspace we will denote $H^1_{\text{bc}}$ for simplicity), suggesting that in fact $\dim H^1_{\text{cusp}} \ll d$ as well.

Their work and ours leads us to think that at a given level, all but finitely many cohomological forms arise from base change from a totally real subfield and CM constructions.  This behaviour differs strongly from that of the spaces $H^1( \Gamma_0(\p), \Q)$ as $\p$ varies, which are expected to contain nontrivial elements for infinitely many $\p$ based on the existence of elliptic curves over the trace field $F$; moreover, this space can be seen to break up under the Hecke action with two dimensional subspaces defined over $\Q$ for each elliptic curve.  The observed behaviour that $H^1_{\text{bc}}$ is equal to $H^1_{\text{cusp}}$ for $d$ large may be seen as analogous to the situation over $\Q$, where Maeda has conjectured that the spaces $H^1( SL(2,\Z), \Sym^d / \Q )$ are irreducible under the Hecke action for $d$ sufficiently large.

\subsection{$p$-Congruence Towers}
\label{coh12}

The proof of theorem \ref{weight} is based on results of Calegari and Emerton on the growth of cohomology with $\F_p$ coefficients in a tower of covers obtained by imposing congruence conditions at a fixed prime $p$.  To describe their results, let $\Gamma \subset SL(2,\C)$ be a general lattice with quaternion algebra $A/F$.  Given a choice of rational prime $p$ and primes $\p_1, \ldots \p_t$ of $F$ above it, we let the closure of the image of the natural embedding $\Gamma \rightarrow \prod A^\times_{\p_i}$ be denoted $K_p$.  If only one prime $\p$ is taken we let the corresponding group be denoted $K_\p$.  If $A$ is split at $\p$ and $K_\p$ is compact we may fix an inclusion $K_\p \subset \SL(2,\OO_{\p} )$, and if $G(\p^k)$, $G_0(\p^k)$ etc. are the standard congruence subgroups of $\SL(2,\OO_\p)$ then the intersection of $\Gamma$ with any of them will be denoted $\Gamma(\p^k)$, $\Gamma_0(\p^k)$, etc.

Consider any choice of $p$, $\p_1, \ldots \p_t$ as above, and let the dimension of the $p$-adic analytic group $K_p$ be $d$. Assuming it is compact, $K_p$ it will contain a uniform pro-$p$ group $G$ of finite index.  If $G = G_1 \rhd G_2 \rhd \ldots$ is the lower $p$-series of $G$, the sequence of sublattices $\Gamma_k = \Gamma \cap G_k$ of $\Gamma$ will be referred to as a $p$-congruence tower.  In \cite{CE2} Calegari and Emerton have proven the following about such towers, which has guided most of our work.

\begin{theorem}
\label{classify}
The $\F_p$ cohomology of a $p$-congruence tower tower exhibits one of the three following types of behaviour:

\begin{list}{ \alph{qcounter}) }{ \usecounter{qcounter} }

\item $H^1(\Gamma_k, \F_p) = \Hom( G_k, \F_p)$ for all $k$, $d=3$ and the pro-$p$ completion of $\Gamma_1$ is analytic,

\item $\dim H^1(\Gamma_k, \F_p) \sim p^{(d-1)k}$,

\item or $\dim H^1(\Gamma_k, \F_p) \sim p^{dk}$.

\end{list}

\end{theorem}

It has been conjectured by them that possibility (c) never happens, which is why we expect to always be able to find a prime for which the hypothesis of theorem \ref{weight} is satisfied.  The essential idea of the proof of theorem \ref{weight} is to use the arithmetic nature of $\Gamma$ to define $\rho_d$ over a $p$-adic field and reduce mod $p$, which converts the problem to one of bounding the $\F_p$ cohomology growth in a tower similar to the ones they consider.  In studying these towers we have proven two results which allow us to determine when the behaviour is of type (a) and (b).  While we only need to prove upper bounds of the form (b) to verify the hypothesis of theorem \ref{weight}, our test for towers of type (a) is interesting in its own right as it simplifies a criterion of Boston and Ellenberg for finding exhaustive towers of rational homology 3-spheres.  Our two classification results are the following.

\begin{prop}
\label{analytic}
Suppose $\p$ is a prime of $F$ of degree 1, $A$ is split at $\p$ and the closure of $\Gamma$ in $A_\p^\times$ is isomorphic to $SL(2,\Z_p)$.  If $\dim H^1(\Gamma(\p), \F_p) \le p-9$ ($p-5$ if $\Gamma$ is congruence), then the congruence tower at $\p$ is of type (a).
\end{prop}

\begin{prop}
\label{saving}
In the notation of theorem \ref{classify}, suppose that for some $k$ we have

\begin{equation*}
\dim H^1(\Gamma_k, \F_p) < (1/d! - 2d p^{-(k-1)} ) | G : G_k |.
\end{equation*}

Then the congruence tower is of type (b).

\end{prop}

Note that the criterion of proposition \ref{saving} will be satisfied whenever the tower is of type (b), although we cannot say how large a $k$ must be checked.  The condition on $K_\p$ that we have assumed in proposition \ref{analytic} is unnecessary, but we have adopted it to make finding examples easier.  Towers of type (a) are of interest topologically, as they satisfy $H^1(\Gamma(\p^k), \Q) = 0$ for all $k$ and thus provide an exhaustion of $\Gamma$ by subgroups with no nontrivial rational cohomology.  They also have the property that the pro-$p$ completion of $\Gamma$ is analytic and equal to the congruence completion at $\p$.  The existence of such towers was first proven by Calegari and Dunfield assuming the Riemann hypothesis and theorems associating Galois representations to automorphic forms over imaginary quadratic fields \cite{CD}, and unconditionally by Boston and Ellenberg shortly after \cite{BE}.  Boston and Ellenberg's criterion is stated below for the case of the tower $\Gamma(\p^k)$ considered in proposition \ref{analytic}, although it may be applied to arbitrary lattices.

\begin{theorem}
If $\Gamma(\p) / \Gamma(\p)^p \simeq (\Z/p\Z)^3$, then the congruence tower is of type (a).
\end{theorem}

Our result is quite similar to theirs, and we should clarify the relationship between the two (let us assume $p > 11$).  Ours is stronger, as the condition $\Gamma(\p) / \Gamma(\p)^p \simeq (\Z/p\Z)^3$ implies that $H^1( \Gamma(\p), \F_p) \simeq \F_p^3$, so our condition is satisfied whenever theirs is.  However, our result implies that if $H^1( \Gamma(\p), \F_p) \simeq \F_p^3$ then the pro-$p$ completion of $\Gamma(\p)$ is equal to the $\p$-congruence completion, and so the condition $\Gamma(\p) / \Gamma(\p)^p \simeq (\Z/p\Z)^3$ will also be satisfied.  Our criterion for analyticity of a pro-$p$ completion is therefore equivalent to that of Boston and Ellenberg, although it cannot be deduced from theirs.  The main difference between them is that ours is easier to computationally verify, and in section \ref{coh42} we summarise computations which suggest that towers of type (a) are relatively common.  We have checked our criterion (modulo an assumption about the nonvanishing of a certain term in a spectral sequence) for a family of hyperbolic 3-orbifolds called the twist knot orbifolds, at split primes with $N(\p) < 1,000$ at which the corresponding lattices are unramified.  For each orbifold we found several $\p$ for which the hypothesis is satisfied, and for the nonarithmetic examples it appears there may in fact be infinitely many such primes.

The structure of the paper is as follows: in section \ref{coh2} we study the cohomology of pro-$p$ groups, proving propositions \ref{analytic} and \ref{saving} as well as lemmas which will be necessary in proving theorem \ref{weight}.  In section \ref{coh3} we prove proposition \ref{vanishing} and theorem \ref{weight}, and in section \ref{coh4} we verify the hypotheses of theorem \ref{weight} for the Bianchi group $SL(2, \OO_{-2})$ and the prime 3, and give experimental data on the number of towers satisfying proposition \ref{analytic}.

\section{$p$-Congruence Towers}
\label{coh2}

This section contains various results about the cohomology of pro-$p$ groups and the structure of their group rings, leading to proofs of propositions \ref{analytic} and \ref{saving}.  The proofs are based on the induction formula $H^1( \Gamma_k, \F_p ) = H^1( \Gamma_1, \F_p[ G / G_k ] )$, combined with the near commutativity of $\F_p[ G / G_k ]$ which allows us to filter this module and deduce results about its cohomology from information about a small quotient.  It will also be necessary for us to switch between the usual cohomology group of cycles modulo boundaries to a space of closed and coclosed chains (a sort of combinatorial `Hodge isomorphism'), both for these propositions and in proving theorem \ref{weight}, and we give the details of this here.  We introduce basic results about pro-$p$ completed group rings in section \ref{coh21}, describe the approximate Hodge isomorphism in section \ref{coh22}, and complete the proofs of the two propositions in the remainder of the secion.  We begin with two small lemmas which will be useful later.

\begin{lemma}
\label{pnormal}
Let $G$ be a group and $N \triangleleft G$ a normal subgroup of index $p$.  Then $\dim H^i(N, \F_p) \le p \dim H^i( G, \F_p)$
\end{lemma}

\begin{proof}
Use the induction formula $H^i(N, \F_p) = H^i( G, \F_p [ G/N ] )$, the filtration of $\F_p [ G/N ]$ by trivial modules and the spectral sequence computing $H^i( G, \F_p [ G/N ] )$ from this filtration.
\end{proof}

\begin{lemma}
\label{congfactor}
Let $\Gamma \subset SL(2,\C)$ be a congruence lattice, so that $\Gamma = A_1^\times(F) \cap K$ for some compact open subgroup of the finite Adele group of $A_1^\times$.  Let $\p$ be a prime of $F$.  If $U_\p \subset A^\times_\p$ is sufficiently small, then there exists a compact open subgroup $U^\p$ of the finite Adeles away from $\p$ such that $\Gamma \cap U_\p = A_1^\times(F) \cap U_\p U^\p$ (i.e. by passing to a deep enough congruence subgroup at $\p$, we can assume that $K$ factors).
\end{lemma}

\begin{proof}
Let $K$ have projections $V_\p$ to $A^\times_\p$ and $V^\p$ to the remainder of the finite Adeles, and choose an open product group $U_\p U^\p \subset K$ which is normal in $V_\p V^\p$.  We then have $K / U_\p U^\p \subset (V_p / U_\p) \times (V^\p / U^\p)$, so if we take $K \cap U_\p$ and let the image of $(K \cap U_\p) / U_\p$ in $V^\p$ be $K^\p$, we see that $K \cap U_\p = U_\p K^\p$ as required.
\end{proof}

\subsection{Pro - $p$ Group Rings}
\label{coh21}

The following results and definitions are taken from \cite{DSMS}.  If $G$ is any pro-$p$ group and $G_k$ a descending exhaustive sequence of open normal subgroups, we may form the projective limit

\begin{equation*}
\Galg := \underset{\longleftarrow}{\lim} \, \F_p[G/G_k],
\end{equation*}

called the completed group algebra of $G$.  This has particularly nice properties when $G$ is uniform, stated in the following theorem.

\begin{theorem}
\label{gpal}
Let $G$ be a uniform pro-$p$ group of dimension $d$ with topological generating set $g_1, \ldots, g_d$.  The completed group ring $\F_p \llbracket G \rrbracket$ is generated by $z_i = 1-g_i$, and every element of it can be uniquely expressed as a sum over multi-indices $\alpha$,

\begin{equation*}
x = \sum_\alpha \lambda_\alpha z^\alpha,
\end{equation*}

where $z^{\alpha} = \Pi_{i=1}^d z_i^{\alpha_i}$ and all $\lambda_\alpha \in \F_p$.  Moreover, all such sums are in $\F_p \llbracket G \rrbracket$.  The filtration by degree gives $\F_p \llbracket G \rrbracket$ the structure of a filtered ring whose associated graded ring is commutative, i.e. $z^\alpha z^\beta = z^{\alpha + \beta}$ up to terms of degree $> |\alpha| + |\beta|$.

\end{theorem}

If $G$ is uniform with lower $p$ series $G = G_1 \rhd G_2 \rhd \ldots$, we may recover the rings $\F_p[G/G_{k+1}]$ as the quotients of $\Galg$ by the kernel $I_k$ of $\Galg \longrightarrow \F_p[G/G_{k+1}]$.  As $G_{k+1}$ is topologically generated by $\{ g_1^{p^k}, \ldots, g_d^{p^k} \}$ we know that $I_{k} = \langle z_1^{p^k}, \ldots, z_d^{p^k} \rangle$, and we denote the corresponding left regular representation by $A_k$.  Because $[G,G] \subset G_2$ we know that all commutators in fact lie in $I_1$, i.e. the group ring is commutative up to terms of degree at least $p-2$ higher.

\subsection{A `Hodge Isomorphism'}
\label{coh22}

For our purposes it will be convenient to approximate $H^1$ by a space which is simply the kernel of a map between two specified $\Gamma$ modules rather than one of the form kernel modulo image; we shall take the space of closed and coclosed 1-chains.  This has the disadvantage of not being intrinsic to $\Gamma$ but makes it easier for us to prove things using the almost-commutativity of $\Galg$.  More precisely, let $\Gamma$ be any finitely presented group and let 

\begin{equation*}
\rightarrow M_2 \overset{\partial}{\rightarrow} M_1 \overset{\partial}{\rightarrow} M_0 \overset{\varepsilon}{\rightarrow} \Z \rightarrow 0
\end{equation*}

be a free resolution of $\Z$ as a $\Z[\Gamma]$ - module, where $\text{rank} \, M_i = n_i$ and we choose $\{ e_j \}_{j=1}^{n_i}$ as free bases.  We may give $\Z[\Gamma]$ the usual inner product

\begin{equation*}
\langle a, b \rangle = \sum_{g \in \Gamma} a_g b_g ,
\end{equation*}

which we extand naturally to $M_i$ and with respect to which we define the coboundary maps $\delta : M_i \rightarrow M_{i+i}$ as the adjoints of $\partial$.  If $A$ is a representation of $\Gamma$ over some vector space $F$ with a nondegenerate invariant bilinear pairing $\langle \;,\; \rangle_A$, we may define pairings on the spaces of $i$-chains $C^i = \Hom_{\Gamma} (M_i, A)$ by

\begin{equation*}
\langle \omega, \eta \rangle = \sum_{j = 1}^{n_i} \langle \omega(e_j), \eta(e_j) \rangle_A.
\end{equation*}

Then the maps induced by $\partial$ and $\delta$ on $C^i$ are adjoints with respect to this pairing, and we have the usual equivalences

\begin{eqnarray*}
\omega \; \text{closed} \Longleftrightarrow \langle \omega, \delta \eta \rangle = 0 \; \text{for all} \; \eta, \\
\omega \; \text{coclosed} \Longleftrightarrow \langle \omega, \partial \eta \rangle = 0 \; \text{for all} \; \eta.
\end{eqnarray*}

Let $\Omega^i(\Gamma, A)$ denote the space of closed and coclosed $i$-chains, and $\Delta^i(\Gamma, A)$ the space of harmonic $i$-chains, i.e. those satisfying $(\delta\partial + \partial\delta)\omega = 0$.  We use the usual notation $Z^i$ and $B^i$ for closed and exact chains, and $Z_i, B_i$ for coclosed and coexact.  The following lemma may be used to show that $\Omega^1$ is a good approximation to $H^1$, where all dimensions are as $F$ vector spaces.

\begin{lemma}
\label{approx}
$| \dim H^1 - \dim \Omega^1 | \le \dim \Delta^0$.
\end{lemma}

\begin{proof}
$B^1 \cap Z_1 \subset \Omega^1$, and has codimension at most $\dim Z^1/B^1 = \dim H^1$.  Therefore $\dim B^1 \cap Z_1 \ge \dim \Omega^1 - \dim H^1$.  However, chains in $B^1 \cap Z_1$ are of the form $\partial \eta$ with $\delta \partial \eta = 0$, i.e. they are boundaries of harmonic $0$-chains.  This gives $\dim \Delta^0 \ge \dim \Omega^1 - \dim H^1$.  For the other direction, $\dim Z_1 + \dim B^1 = \dim C^1$ by the nondegeneracy of our inner product on chains, and $ \dim B^1 \cap Z_1 \le \dim \Delta^0$.  Therefore the image of $Z_1$ in $C^1 / B^1$ has codimension at most $\dim \Delta^0$, and likewise for the image of $\Omega^1 = Z^1 \cap Z_1$ in $H^1$.
\end{proof}

As $\Omega^1$ and $\Delta^0$ depend on the resolution of $\Z$, we shall fix a resolution (depending on a presentation of $\Gamma$) in which $\Delta^0$ will turn out to be small in all cases we consider.  Let

\begin{equation*}
\Gamma = \langle g_1, \ldots g_n | R_1, \ldots, R_m \rangle, \quad R_i = \prod_{j=1}^{l_i} g_{\alpha(i,j)}
\end{equation*}

be a presentation of $\Gamma$.  Then we may choose the last three terms in our complex to be 

\begin{equation*}
\longrightarrow \Z[\Gamma]^m \overset{\partial_2}{\longrightarrow} \Z[\Gamma]^n \overset{\partial_1}{\longrightarrow} \Z[\Gamma] \longrightarrow 0,
\end{equation*}

where $\partial_1$ and $\partial_2$ are given by 

\begin{equation*}
\partial_1(e_i) = 1-g_i, \quad \partial_2(e_i) = \sum_{j=1}^{l_i} \prod_{k=1}^{j-1} g_{ \alpha(i,k) } e_{ \alpha(i,j) }.
\end{equation*}

Subject to these choices, the operator $\Delta = \delta \partial$ on the space $C^1 = \Hom_{\Gamma} ( \Z[\Gamma], A) \simeq A$ is given by

\begin{eqnarray*}
\Delta & = & \left( \sum_{i=1}^n 2 - g_i - g_i^{-1} \right) \\
 & = & \left( - \sum_{i=1}^n g_i^{-1}( 1-g_i )^2 \right)
\end{eqnarray*}

\subsection{Proof of Proposition \ref{saving} }
\label{coh23}

Recall the assumptions of proposition \ref{saving}, in which $\Gamma$ injected densely into a $p$-adic analytic group $K_p$ with uniform pro-$p$ subgroup $G$ of dimension $d$.  We defined $G = G_1 \rhd G_2 \rhd \ldots$ to be the lower $p$-series of $G$, and set $\Gamma_k = \Gamma \cap G_k$.  We shall apply the discussion of section \ref{coh22} to the group $\Gamma_1$ and the module $\F_p[ \Gamma_1 / \Gamma_{k+1} ] \simeq A_k$, from which we will deduce information about $H^1(\Gamma_{k+1}, \F_p)$ by the induction formula.  We first modify the generating set of $\Gamma_1$ so that the first $d$ generators $\{g_1, \ldots g_d \}$ form a topological generating set for $G$ and the remainder lie in $G_2$.  Choose the $z_i$ in the structure theorem for the completed group ring so that $z_i = 1-g_i$ for $1 \le i \le d$; then $1-g_i$ has degree at least $p$ for $i > d$, and the expression for $\Delta$ becomes

\begin{equation*}
\Delta = \left( -\sum_{i=1}^d z_i^2 + \text{higher order terms} \right).
\end{equation*}

Subject to this choice we have $\dim \Delta^0 \le 2d p^{(d-1)k}$, as may be seen by considering the subspace $V$ of $A_{ k }$ given by $V = \text{span} \{ z^{\alpha} | \alpha_i \le p^{k} - 3 \}$.  $\Delta$ clearly has no kernel on $V$ by the grading on $\F_p \llbracket G \rrbracket$, and its dimension can be seen to be at least $(1 - 2d p^{-k}) \dim A_{ k }$ which gives the desired bound.

The main idea of the proof is to give an effective upper bound for $\Omega^1$, or more generally the kernel of any map $\overline{T}_{ k } : M_{ k } \rightarrow N_{ k }$ obtained by reducing mod $I_k$ a map $T : M \rightarrow N$ between free $\F_p \llbracket G \rrbracket$ modules.  We may think of $T$ as left multiplication by a matrix in $\text{Mat}( \F_p \llbracket G \rrbracket )$.  As the finite quotients of $\F_p \llbracket G \rrbracket$ are group algebras and hence self-dual, it is equivalent to bound the codimension of the image of $\overline{T}_{ k }^+ : N_{ k } \rightarrow M_{ k }$ where taking adjoints consists of transposing a matrix and inverting every group element appearing in its entries.

\begin{lemma}
\label{quotient}
Suppose $T : M \rightarrow N$ is such a map, with quotients $\overline{T}_{ k } : M_{ k } \rightarrow N_{ k }$, and suppose that for some $k_0$ we have

\begin{equation*}
\dim \text{coim}\, \overline{T}_{ k_0 } < p^{dk_0}/d!
\end{equation*}

Then for any $k$ we have

\begin{equation*}
\dim \text{coim}\, \overline{T}_{ k } \le  dp^{k_0 - k} \dim N_{k} .
\end{equation*}
\end{lemma}

\begin{proof}

Let $v_i$ be a free basis for $N$, and let $D \subset \F_p \llbracket G \rrbracket$ be the subspace spanned by all monomials of degree less than $p^{k_0}$.  Then $\dim D \ge p^{dk_0}/d!$, and by our assumption $\im \overline{T}_{ k_0 }$ contains an element of $D v_i$ for each $i$ which we denote $x_i$.  As $\overline{T}_{ k_0 }$ commutes with right multiplication by $\F_p \llbracket G \rrbracket$, we may assume that all $x_i$ have degree $p^{k_0}-1$.  Let $N^0 \supset N^1 \supset N^2 \ldots$ be the filtration on $N$ by degree, and $N^i_{ k }$ the corresponding filtration on $N_{k}$.  Lifting the $x_i$ to $N$, we obtain elements $\{ x_i \}$  of $\im T$ of degree $p^{k_0}-1$ such that the image of $x_i$ in $N^{p^{k_0}-1} / N^{p^{k_0}}$ is supported in the $i$th co-ordinate.  For a given $k$, define $P_d$ by

\begin{equation*}
P_d = \text{span} \{ z^\alpha | \Sigma \alpha_i = d-p^{k_0}+1, \, \alpha_i \le p^{k} - p^{k_0} \}.
\end{equation*}

Because $T$ commutes with right multiplication by $\F_p \llbracket G \rrbracket$, for each $i$ we have a map $P_d \rightarrow \im T$ given by $p \mapsto x_i p$, and by the almost commutativity of $\F_p \llbracket G \rrbracket$ we see that the map $P_d \rightarrow ( \im \overline{T}_{ k }\cap N^d_{ k } ) / N^{d+1}_{ k }$ is an injection.  Therefore

\begin{eqnarray*}
\dim \im \overline{T}_{ k } & \ge & \dim \bigcup P_d \\
\dim \im \overline{T}_{ k } & \ge & \dim N_{k } ( 1 - dp^{k_0 - k} )
\end{eqnarray*}

as required.

\end{proof}

To deduce proposition \ref{saving} from this, we know that $| \dim H^1(\Gamma, A_k) - \dim \Omega^1( \Gamma, A_k) | \le 2d p^{(d-1)k}$, so the condition of the theorem implies a bound on $\dim \Omega^1(\Gamma, A_{k-1})$ sufficient to apply lemma \ref{quotient}.  This gives us a power saving in the growth of $\Omega^1$, and hence a power saving in $H^1$ as required.\\

\subsection{Proof of Proposition \ref{analytic} }
\label{coh24}

The proof of proposition \ref{analytic} uses much the same ideas as that of proposition \ref{saving}, although we have to work a little harder because we will apply them to a group which is not uniform and hence whose group algebra is not commutative up to higher order terms (this is only for practical reasons, as noted in the introduction).  For the rest of section \ref{coh3}, $G_i$ will denote the principal congruence subgroup of $SL(2,\Z_p)$ of level $p_i$ and we set $G = G_1$.  Instead of using the near-Abelianness of $\Galg$ to provide a power saving with respect to a uniformly shrinking family of congruence subgroups, we shall shrink the group in one direction at a time and thus be able to prove power saving of the form $\dim H^1( \Gamma_k , \F_p ) \ll p^k$, which forces case (a) to hold.  The families of subgroups we will descend along are as follows:

\begin{eqnarray*}
H_k = \left\{ \left( \begin{array}{cc} 1 + pa & p^k b \\ pc & 1+pd \end{array} \right) | a, b, c, d \in \Z_p \right\},\\
P_{k,l} = \left\{ \left( \begin{array}{cc} 1 + pa & p^k b \\ p^l c & 1+pd \end{array} \right) | a, b, c, d \in \Z_p \right\},
\end{eqnarray*}

where we set $H = H_0$.  Let $\Lambda = \Gamma \cap H$, $\Lambda_k = \Gamma \cap H_k$, and $\Lambda_{k,l} = \Gamma \cap P_{k,l}$ .  We shall first prove that $\dim H^1( \Lambda_k, \F_p ) \le p-5$, then $\dim H^1( \Lambda_{k,l}, \F_p ) \le p-1$, and finally that $\dim H^1( \Gamma_k , \F_p ) \ll p^k$ by lemma \ref{pnormal}.  $H$ is not uniform, and so we shall first need some results on $\Halg$ to stand in for theorem \ref{gpal}.

\subsection{Structure of $\Halg$ }
\label{coh25}

Choose a topological generating set $\{ g_1, g_2, g_3\}$ for $H$ whose elements are upper triangular, diagonal and lower triangular respectively.  This is not a minimal generating set, but the Bruhat decomposition does ensure that every element of $H$ has a unique expression as $g = g_1^a g_2^b g_3^c$ for $p$-adic integers $a, b, c$.  This means that the completed group algebra $\Halg$ stil has the feature that if we choose generators $z_i = 1 - g_i$ as before then every element will have a unique expression as 

\begin{equation*}
x = \sum_\alpha \lambda_\alpha z^\alpha.
\end{equation*}

The derived subgroup of $H$ is generated by $g_1^p, g_2$ and $g_3^p$, and correspondingly elements of $\Halg$ commute up to terms in the ideal $D = \langle z_1^p, z_2, z_3^p \rangle$.  Moreover, because $g_1$ and $g_2$ generate a uniform subgroup, $z_1$ and $z_2$ commute up to terms which are polynomials of degree at least $p$ in $z_1$ and $z_2$, and likewise for $z_2$ and $z_3$.  As before, for $k \ge 0$ define the ideal $I_k $ of $\Halg$ corresponding to the subgroup $G_{k+1}$ by

\begin{equation*}
I_k = \langle z_1^{p^{k+1}}, z_2^{p^k}, z_3^{p^k} \rangle.
\end{equation*}

Let $N_k$ be the coset representation $\F_p[H/H_k]$ of $\Lambda$ so that $H^1( \Lambda_k, \F_p ) = H^1( \Lambda, N_k )$.  $\F_p[H/H_k]$ is the quotient of $\Halg$ by the left ideal $J + I_{k-1}$, where $J = \langle z_2, z_3 \rangle$.  $N_k$ may be thought of as quotients of $N = \Halg / J$.  We will need the following fact about $J$:

\begin{lemma}
\label{ideal}
As a $\F_p$ vector space, the right ideal $J = \text{span} \{ z^n | n_2+n_3 \ge 1 \}$.

\end{lemma}

\begin{proof}
Let $W$ be the vector space on the right.  As $J$ is generated by $z_2$ and $z_3$ we know that $W \subset J$, and to show equality it is enough to show that $W$ is a right ideal.  To show that $z_i z^n \in W$ for a basis monomial $z^n$ of $W$ and a generator $z_i$ of $\Halg$, express $z_i z_1^{n_1}$ in standard form so that $z_i z^n$ may be written as $\sum_\alpha \lambda_\alpha z^\alpha z_2^{n_2}z_3^{n_3}$.  As discussed, any monomial $z_2^{\alpha_2} z_3^{\alpha_3} z_2^{n_2} z_3^{n_3}$ can be expressed as $\sum_\beta \mu_\beta z_2^{\beta_2} z_3^{\beta_3}$ with $|\beta| \ge n_2 + n_3 + \alpha_2 + \alpha_3 \ge 1$, so the basis monomials expressing $z_i z^n$ are in $W$ and it is an ideal.

\end{proof}

Lemma \ref{ideal} shows that $N_k$ is naturally isomorphic to $\F_p[z_1]/( z_1^{p^k} )$ (although not as $\Halg$ modules).  The following near - Abelianness statement for the action of $\Halg$ on $N_k$ will play a similar role as theorem \ref{gpal} did in the proof of proposition \ref{saving}.

\begin{lemma}
\label{abel}
For $v = p(z_1) \in N_k$ of degree $d$ and $x \in \Halg$, $(x z_1^t) v = (z_1^t x) v + z_1^{t+d+p-2} N_k$.
\end{lemma}

\begin{proof}
As $z_1^t x ( v + (J+I_{k-1}) ) = z_1^t x v + (J+I_{k-1})$ and $x z_1^t ( v + (J+I_{k-1}) ) = x z_1^t v + (J+I_{k-1})$, it suffices to show that $[x, z_1^t] z_1^d \in J + z_1^{t+d+p-2} \Halg$.  At this is linear in $x$ and clearly satisfied for $x = z_1^n$, we only need to show that $z_1^a x z_1^b \in J + z_1^{a+b+p-2} \Halg$ for $x \in J$ a basis monomial.  Define the function $f$ on all monomials (i.e. not necessarily in the standard $z_1^{\alpha_1} z_2^{\alpha_2} z_3^{\alpha_3}$ form) as the total degree plus $p-2$ times the degree of $z_2$.  Then our discussion of the structure of $\Halg$ implies that when a monomial is redueced to standard form under the commutation relations, $f$ grows monotonically on the terms produced.  Consequently, the following two kinds of monomials will lie in $J + z_1^{a+b+p-2} \Halg$ when reduced to standard form:

\begin{list}{ \alph{qcounter}) }{ \usecounter{qcounter} }

\item Terms of degree $\ge a + b$ in which $z_2$ occurs at least once.

\item Terms of degree $\ge a + b + p-2$.

\end{list}

The only other possibility is if $x$ is of the form $z_1^{n_1} z_3^{n_3}$, but all nontrivial terms generated when a $z_3$ and $z_1$ are commuted past each other will be of one of the two types above and so also be reduced to lie in $J + z_1^{a+b+p-2} \Halg$.

\end{proof}

\subsection{Conclusion of Proof}
\label{coh26}

We may now proceed as before, by choosing a generating set $\{ g_1, \ldots g_n \}$ for $\Lambda$ in which all elements other than $g_1$ lie in $G_1$.  Then $g_i^{-1}(g_i - 1)^2 \in J + I_0 \subset \Halg$ for $i > 1$, and $g_1^{-1}(g_1 - 1)^2 = z_1^2 + z_1^3 \Halg + J$.  If we choose a resolution of $\Z$ as before, the action of $\Delta^0$ on $C^0 \simeq N_k$ will be by multiplication by something of the form $-z_1^2 + z_1^3 \Halg + J$, and by lemma \ref{abel} this implies

\begin{equation*}
\Delta^0( z_1^d + J + I_{k-1} ) = -z_1^{d+2} + z_1^{d+3} \Halg + J.
\end{equation*}

Therefore $\dim \Delta^0 \le 2$ for all $k$, and lemma \ref{approx} gives us that $| \dim H^1( \Lambda, N_k ) - \dim \Omega^1( \Lambda, N_k) | \le 2$.  Combined with the assumption of the proposition, we have $\dim \Omega^1( \Lambda, N_1) \le p-7$.  We now give an analog of lemma \ref{quotient} which proves that $\dim \Omega^1( \Lambda, N_k) \le p-7$ for all $k$ based on this assumption.

\begin{lemma}
If $\dim \Omega^1( \Lambda, N_1) = r \le p-3$, then $\dim \Omega^1( \Lambda, N_k) \le r$ for all $k$
\end{lemma}

\begin{proof}

Let $T = \partial \oplus \delta: L \rightarrow M$, where $L = C^1(\Lambda, N)$ and $M = C^0(\Lambda, N) \oplus C^2(\Lambda, N)$ are isomorphic to $N^m$ and $N^n$ for some $m$ and $n$.  Denote the quotient modules $C^1(\Lambda, N_k)$ and $C^0(\Lambda, N_k) \oplus C^2(\Lambda, N_k)$ by $L_k$ and $M_k$, and the induced map between them by $T_k$.  $L$ and $M$ are free $\F_p[z_1]$ modules, and while $T$ does not preserve this structure by lemma \ref{abel} it does preserve the filtration by degree and commutes with the action of $\F_p[z_1]$ up to terms of degree $p-2$ higher; in other words, the induced maps on the subquotients 

\begin{equation}
\label{subquot}
\overline{T}_d : z_1^d L / z_1^{d+p-2} L \rightarrow z_1^d M / z_1^{d+p-2} M
\end{equation}

 preserve the $\F_p[z_1]$ structure.  If we choose free bases $\{ u_1, \ldots, u_m \}$ and $\{ v_1, \ldots, v_n \}$ for $L$ and $M$, $\overline{T}_d$ is given by a fixed matrix $A$ in $\text{Mat} ( \F_p[z_1] / (z_1^{p-2}) )$ and by altering our bases suitably we may assume that $A$ is in reduced form, i.e. that it is diagonal with entries of increasing degree.  When $d=2$, the map (\ref{subquot}) gives the action of $T_1$ on the submodule $z_1^2 L_1$ of $L_1$, and by assumption the kernel of this operator has dimension at most $ r \le p-3$.  This implies that $\dim \ker A \le r$, and in particular all the diagonal entries of $A$ must have degree at most $p-3$ so that the kernel of (\ref{subquot}) must be contained in $z_1^{d+1} L / z_1^{d+p-2} L$.  If $w \in L_k$ is in the kernel of $T_k$, considering this with $d$ the degree of $w$ shows that $d \ge p^k - p + 3$ and hence $\ker T_k \subset z_1^{p^k-p+3} L_k$.  This kernel naturally lies inside the kernel of $\overline{T}_{p^k - p + 2} \simeq A$, and so is bounded above by $\dim \Omega^1(\Lambda, N_1) = r$.

\end{proof}

Passing back to cohomology, we have $\dim H^1( \Lambda_k, \F_p) \le p-5$ for all $k$.  Now for fixed $k$ consider the lattice $\Lambda_{k,0}$ and its family $\Lambda_{k,l}$ of subgroups.  $P_{k,l}$ is conjugate to $H_{k+l-1}$ in $\text{GL}(2,\Z_p)$ and $P_{k,l}$ is even uniform for $k+l \ge 2$, so we may work exactly as before using the input $\dim H^1( \Lambda_{k,1}, \F_p) \le p-5$ to show $\dim H^1( \Lambda_{k,l}, \F_p) \le p-1$ for all $k$ and $l$.  Finally, lemma \ref{pnormal} allows us to pass from $\dim H^1( \Lambda_{k,k}, \F_p) \ll 1$ to $\dim H^1( \Gamma_k, \F_p ) \ll p^k$, which completes the proof.  Note that if $\Gamma$ is congruence we may use elements of its commensurator to `round out' the lattices $\Lambda_{2k}$ to ones which are close to $\Lambda_{k, k}$ (i.e. so that their intersection remains of bounded index in both), and so derive the bound $\dim H^1( \Lambda_{k,k}, \F_p) \ll 1$ by lemma \ref{pnormal}.  This allows us to dispense with this second descent argument and prove the theorem using only an initial bound of $\dim H^1( \Gamma_1, \F_p ) \le p-5$.

\section{Power Savings in the Weight Aspect}
\label{coh3}

We will now prove proposition \ref{vanishing} and theorem \ref{weight}, which restrict the allowable weights of cohomological automorphic forms and provide a conditional upper bound for their multiplicities.  Throughout, $A$ and $F$ will be used to denote the trace field and quaternion algebra of the stably arithmetic lattice under consideration.  We begin with the following lemma, which allows us to change between representations of $\Gamma$ defined at different infinite or finite places of the trace field.

\begin{lemma}
\label{placechange}
Let $A/F$ be a quaternion algebra, $K/F$ an extension over which $A$ splits, and $\sigma: K \rightarrow K_w$ an embedding of $K$.  Then the module for $A$ over $K_w$ defined by the corresponding embedding

\begin{equation*}
A \longrightarrow A \otimes K \simeq M_2(K) \overset {\sigma} {\longrightarrow} M_2(K_w)
\end{equation*}

depends only on the embedding $\sigma_0: F \rightarrow F_v$ of $F$ below $\sigma$, and is isomorphic to the two dimensional module given by the embedding $A \longrightarrow A_v$.
\end{lemma}

\begin{proof}
The above map $A \longrightarrow M_2(K_w)$ factors as $A \longrightarrow A_v \longrightarrow A_v \otimes_{F_v} K_w$, and so the two dimensional irreducible module for $M_2(K_w)$ restricts to a two dimensional irreducible module for $A_v$ in which $F_v$ acts by the embedding $\sigma_0$, which is unique.
\end{proof}

\subsection{Proof of Proposition \ref{vanishing}}
\label{coh31}

To prove proposition \ref{vanishing}, we choose an extension $L$ over which $A$ splits and embeddings $\widetilde{\sigma}_1, \ldots, \widetilde{\sigma}_n$ of $L$ which extend the archimedean embeddings $\sigma_1, \ldots, \sigma_n$ of $F$.  With notation as in the introduction, by lemma \ref{placechange} $\rho_d$ may be expressed as

\begin{equation*}
\rho_d = \bigotimes_{i=1}^n \Sym^{d_i} \circ \widetilde{\sigma}_i
\end{equation*}

Let $\widetilde{L}$ be the Galois closure of $L$.  We may calculate $H^i(\Gamma, \rho_d)$ using a resolution of $\widetilde{L}$ as a $\widetilde{L}[\Gamma]$ module, and so by acting on cochains an element $g$ of $\text{Gal}(\widetilde{L})$ induces an isomorphism between the cohomology groups $H^i(\Gamma, \rho_d)$ and $H^i(\Gamma, g(\rho_d) )$.  Here $g(\rho_d)$ is the representation given by applying $g$ to the matrices of $\rho_d$, and is equal to

\begin{eqnarray*}
g(\rho_d) & = & \bigotimes_{i=1}^n \Sym^{d_i} \circ \widetilde{\sigma}_i \circ g \\
& = & \bigotimes_{i=1}^n \Sym^{d_{g(i)} } \circ \widetilde{\sigma}_i,
\end{eqnarray*}

where $g(i)$ is the automorphism of the archimedean embeddings of $F$ associated to $g$.  The action of $g$ commutes with the restriction maps to the cohomology of parabolic subgroups of $\Gamma$, and so also gives an isomorphism between the cuspidal subspaces.  Therefore if $H^i_{\text{cusp}} ( \Gamma, \rho_d ) \neq 0$, by applying the weight condition of Borel and Wallach to $\rho_d \circ g$ we see that $d_{g(i)}$ and $d_{g(j)}$ must be equal whenever $\sigma_i$ and $\sigma_j$ are a pair of conjugate complex embeddings.  Let $K$ be the Galois closure of $F$ with Galois group $G$, and $H \subset G$ the stabiliser of $F$.  If $\sigma$ is an embedding of $K$, the embeddings of $F$ are given by $\sigma \circ g_i$ for $g_i \in G / H$ a set of coset representatives, and the weight $d$ is an element of $\Z[G / H]$.  If $\tau \in G$ is complex conjugation with respect to $\iota$ then the condition that $d_i$ and $d_j$ be equal for $\sigma_i = \overline{\sigma}_j$ is equivalent to saying that $d$ is invariant under left multiplication by $\tau$, and the condition that $d_{g(i)}$ and $d_{g(j)}$ be equal for all $g$ is equivalent to saying that $d$ is invariant under the normal closure $N$ of $\tau$.  As the fixed field of $N$ is the maximum totally real subfield of $K$, this implies that the fixed field of $HN$ is the maximum totally real subfield $F_0$ of $F$.  As $d$ is right invariant under $HN$, this implies that that value of $d_i$ is equal at all embeddings above a given embedding of $F_0$.

\subsection{Proof of Theorem \ref{weight} }
\label{coh32}

We shall now prove theorem \ref{weight}.  Recall that our approach will be to use the arithmetic nature of $\Gamma$ to define the representations $\rho_d$ over $\Q_p$ rather than $\C$.  This will allow us to reduce them modulo $p$ and bound $H^1(\Gamma, \rho_d)$ in terms of cohomology with $\F_p$ coefficients in the level aspect, which we can control using the tools of section \ref{coh2}.

Let $K$ be a Galois extension which splits $A$, $G = \text{Gal}(K/\Q)$ and $H$ be the stabiliser of $F$ in $G$.  Let $\sigma$ be a complex embedding of $K$ extending the canonical one of $F$, so that the other archimedean embeddings of $F$ are given by $\sigma \circ g_i$ for $g_i \in G/H$ a set of coset representatives.  As in the proof of theorem \ref{vanishing} we may express $\rho_d$ over $K$ by the diagram

\begin{equation*}
A^\times \longrightarrow GL(2,K) \overset {\bigoplus g_i } {\longrightarrow} \bigoplus GL(2,K) \longrightarrow GL(N,K),
\end{equation*}

where the last arrow is given by the tensor product of the $d_i$th symmetric power homomorphisms.  Composing this map with $\sigma$ would give the usual realisation of $\rho_d$ over $\C$, and by choosing instead a totally split $p$-adic place $v$ of $K$ we obtain a representation $\rho_v : A^\times \rightarrow GL(N,\Q_p)$.  In this case the collection $\{ v_i = v \circ g_i \}$ is a complete set of places of $F$ above $p$, and by lemma \ref{placechange} $\rho_v$ has an alternative expression via the direct sum of the embeddings of $A$ into its completions $A_{v_i}$ as

\begin{equation*}
A^\times \longrightarrow \bigoplus A_{v_i}^\times \simeq GL(2,\Q_p)^n \longrightarrow GL(N,\Q_p).
\end{equation*}

We have

\begin{equation*}
\dim_\C H^1(\Gamma, \rho_d) = \dim_{\Q_p} H^1( \Gamma, \rho_v ).
\end{equation*}

\subsection{Reduction to Positive Characteristic}
\label{coh33}

Let $G_i$ be the standard maximal compact subgroups of $A_{v_i}^\times$, and let $G_i(p)$ be their congruence subgroups of first full level.  For simplicity, we shall now assume that $\Gamma(p)$ is dense in $\prod G_i(p)$.  If not, we may carry out the proof below with minor modifications after passing to $\Gamma(p^t)$ for which density holds.  We will still have the power saving hypothesis by lemma \ref{pnormal}, and will restrict $\rho_v$ to $\Gamma(p^t)$ and take $L$ to be the polynomials which are integrally valued on $G(p^t)$.

From now on, we shall think of $\rho_v$ as a representation of $\Gamma(p)$ by restriction.  As discussed, we shall find a lattice $L \subset \rho_v$ such that $L / p L$ factors through a suitable congruence subgroup of $\Gamma(p)$.  It suffices to find lattices $L_i$ in the representations $\Sym^{d_i}$ of $G_i(p)$, as we may obtain $L$ from these by tensoring and restricting to $\Gamma(p)$.  As $G_i$ are all isomorphic we shall denote them by $G$, with standard congruence subgroups by $G(p^k), G_0(p^k)$ etc.

We shall realise the representation $\Sym^d$ of $G(p)$ on the space of homogeneous polynomials $p: \Z_p^2 \rightarrow \Q_p$ of degree $d$ (strictly speaking this is the symmetric power of the dual of the standard representaton, but the standard representation is self-dual), with the action given by $(g \circ f)(x) = f( g^{-1}x )$.  There is an intertwining map from this space to the left regular representation of $SL_2(\Z_p)$ on $C( SL_2(\Z_p), \Q_p)$ given by

\begin{equation*}
p \mapsto ( g \mapsto f(g v_1))
\end{equation*}

where $v_1$ is the first standard basis vector, whose image is the space of functions

\begin{equation*}
f: \left( \begin{array}{cc} a & b \\ c & d \end{array} \right) \mapsto p(a,c)
\end{equation*}

for $p$ a homogeneous polynomial of degree $d$.  All such functions are invariant from the right under the upper triangular subgroup and transform under the diagonal according to

\begin{equation*}
f\left[ g \left( \begin{array}{cc} a & 0 \\ 0 & a^{-1} \end{array} \right) \right] = a^df(g),
\end{equation*}

and so the image of our intertwining map is actually contained inside a representation induced from a character of the Borel subgroup.  We shall choose a lattice in $\Sym^d$ consisting of those polynomials whose values on $G(p)$ are integral, which we denote by $L$.  The reduced lattice $L/pL$ is then seen to be isomorphic to the space of $\F_p$ valued functions on $G(p)$ given by the reduction of integral valued polynomials in $a$ and $c$, and Lucas' theorem shows that such functions must be invariant under $G(p^{k+1})$ for $k$ satisfying $p^k \ge d$.

The invariance under $G(p^{k+1})$, combined with the behaviour under right translation by the Borel, shows that $L/pL$ may be realised as a submodule of the right coset representation $\F_p [ G(p) / ( G(p) \cap G_0(p^{k+2}) ) ]$.  Choose lattices $L_i$ inside the representations $\Sym^{d_i} \circ v_i$ in this way.  Forming the lattice $L \subset \rho_v$ as their tensor product, we see that $L/pL$ may be realised as a subspace of the representation of $\Gamma(p)$ on the module

\begin{eqnarray*}
V & = & \bigotimes \F_p [ G(p) / ( G(p) \cap G_0(p^{k_i+2}) ) ] \\
 & = & \F_p [ \Gamma(p) / ( \Gamma(p) \cap \Gamma_0(\Pp) ) ],
\end{eqnarray*}

where $\Pp = \prod_{i=1}^n \p_i^{k_i+2}$.  Note that we have used the assumption of $\Gamma(p)$ being dense in $\prod G_i(p)$ here.  As

\begin{equation*}
\dim_{\Q_p} H^1( \Gamma(p), \rho_v ) \le \dim_{\F_p} H^1( \Gamma(p), L/p L ),
\end{equation*}

we should be able to get from this a bound of the form

\begin{eqnarray*}
\dim_{\Q_p} H^1( \Gamma(p), \rho_v ) & \le & \dim_{\F_p} H^1( \Gamma(p), V )\\
& = & \dim_{\F_p} H^1( \Gamma(p) \cap \Gamma_0(\Pp), \F_p ).
\end{eqnarray*}

However, to do this we either need to control $H^0$ of the quotient of these two representations or pass to $\dim \Omega^1$ (for which such an inequality clearly holds on passing to subrepresentations) as an intermediate step.  We shall do the latter, using the methods of section 1.  If we let $K = \sum k_i$ and $|k| = \max k_i$, it can be seen that after making the appropriate choice of generators for $\Gamma(p)$ we have a bound 

\begin{equation*}
\dim \Delta^0( \Gamma(p), L/p L ) \le \dim \Delta^0( \Gamma(p), V ) \ll p^{K - |k|},
\end{equation*}

and hence an inequality of the required type,

\begin{equation*}
\dim_{\Q_p} H^1( \Gamma(p), \rho_v ) \le  \dim_{\F_p} H^1( \Gamma(p) \cap \Gamma_0(\Pp), \F_p ) + O( p^{K - |k|} ).
\end{equation*}

As $p^K \sim \prod d_i$, the theorem will then follow from an estimate of the form

\begin{equation}
\dim H^1( \Gamma(p) \cap \Gamma_0(\Pp), \F_p ) \ll p^{(1-\delta)K}.
\end{equation}

Note that it suffices to prove bounds of the form

\begin{equation}
\label{gamma0}
\dim H^1( \Gamma(p) \cap \Gamma_0(\p_i^k), \F_p ) \ll p^{(1-\delta)k}
\end{equation}

for each prime $\p_i$, as we may deduce the more general statement from this by taking $\p_i$ to be the prime for which $k_i$ is largest and applying lemma \ref{pnormal}.

\subsection{Power Savings in Skewed Towers}
\label{coh34}

This section is devoted to the proof of (\ref{gamma0}), which proceeds deriving a power saving in the skewed tower $\Gamma(p) \cap \Gamma_0(\p_i)$ from the saving

\begin{equation}
\label{roundtower}
\dim H^1( \Gamma(p) \cap \Gamma(\p_i^k), \F_p) \ll p^{2k}
\end{equation}

in a uniform tower given by the hypotheses of the theorem and the theorem of Calegari and Emerton.  We assume that $\Gamma = \Gamma(p)$ and setting $\p = \p_i$, and let $G_i$ be the principal congruence subgroup of $G = SL(2,\Z_p)$ of level $p^i$.  As the heart of the proof is in manipulating a variety of $\p$-congruence subgroups of $\Gamma$, for any $U_\p \subset G$ we shall use $\Gamma(U_\p)$ to denote denote the intersection of $\Gamma$ with $U_\p$.  Let $\widetilde{H}^1(\Gamma, \F_p )$ denote the injective limit of the cohomology groups $H^1( \Gamma(\p^k), \F_p)$.  This space carries an action of the congruence completion of $\Gamma$, hence of $G_1$, and for $U_\p \subset G_1$ we define

\begin{equation*}
\widehat{H}^1( \Gamma(U_\p), \F_p ) = \widetilde{H}^1( \Gamma, \F_p )^{U_\p}.
\end{equation*}

The groups $\widehat{H}^1( \Gamma(U_\p), \F_p )$ also satisfy the power saving (\ref{roundtower}).  If a cohomology class on $\Gamma(U_\p)$ vanishes on restriction to some smaller subgroup $\Gamma(G_k)$ then it must descend to a cohomology class on $\Gamma(U_\p) / \Gamma(G_k) \simeq U_\p / G_k$, and as the space of such classes is three dimensional we have

\begin{equation}
\label{complete}
\dim \widehat{H}^1( \Gamma(U_\p), \F_p ) \le \dim H^1( \Gamma(U_\p), \F_p ) + 3.
\end{equation}

We will need to use the commensurator of $\Gamma$ to provide isomorphisms between the cohomology of different congruence subgroups.  By lemma \ref{congfactor} there is a $t$ such that $\Gamma(G_t)$ is of the form $A^\times_1(F) \cap ( G_t U^\p )$.  Therefore if $\eta \in A^\times(F)$ is any rational element which normalises $U^\p$, and $V_\p$ is such that $V_\p, \eta V_\p \eta^{-1} \subset G_t$, then $\eta \Gamma( V_\p ) \eta^{-1} = \Gamma( \eta V_\p \eta^{-1} )$.  We therefore have 

\begin{equation}
\label{conj}
H^1( \Gamma( V_\p ), \F_p ) \simeq H^1( \Gamma( \eta V_\p \eta^{-1} ), \F_p )
\end{equation}

and likewise for $\widehat{H}^1$.  The group $\Lambda$ of such elements in the commensurator will be quite large; in particular its closure in $A_\p^\times$ will be of finite index, so if we fix $A_\p \simeq M_2(\Q_p)$ there will be some $r$ such that

\begin{equation*}
g_r = \left( \begin{array}{cc} p^r & 0 \\ 0 & 1 \end{array} \right)
\end{equation*}

is in this closure.  Because conjugation is continuous we see that (\ref{conj}) holds with $\eta = g_r$.

We turn to the main argument, and introduce a family of subgroups of $G_1$ on which to perform induction.  Let $T \subset G_1$ be the diagonal subgroup, $l$ and $m$ be natural numbers satisfying $l \ge m$ to be fixed later, and for any $k$ and $i$ define the subgroups $A_{i,k}$ by

\begin{equation*}
A_{i,k} = \{ a \in T \; | \; a \cong 1 \; (p^{k-il}) \}.
\end{equation*}

Let

\begin{equation*}
n_i = \left( \begin{array}{cc} 1 & p^{il-m} \\ 0 & 1 \end{array} \right), \qquad n'_k = \left( \begin{array}{cc} 1 & p^{k-m} \\ 0 & 1 \end{array} \right),
\end{equation*}

and let the subgroup generated by $n'_k$ be $N'_k$.  Let $A_{i,k}' = n_i A_{i,k} n_i^{-1}$, so that $A_{i,k}'$ is generated by the element

\begin{equation*}
\left( \begin{array}{cc} 1+p^{k-il} & 0 \\ 0 & (1+p^{k-il})^{-1} \end{array} \right) \left( \begin{array}{cc} 1 & u \cdot p^{k-m} \\ 0 & 1 \end{array} \right)
\end{equation*}

for some unit $u$, and hence $A_{i,k}A_{i,k}' = A_{i,k}N'_k$.  We shall prove by induction on $i$ that

\begin{equation}
\label{induct}
\dim \widehat{H}^1(\Gamma( G_k A_{i,k} ), \F_p )  \le  C \alpha^i p^{2k}
\end{equation}

for all $i$ such that $A_{i,k} \subset G_1$, and some uniform $C$ and $\alpha < 1$.  By setting $i$ to be the appropriate constant multiple of $k$ we will obtain

\begin{equation*}
\dim H^1( \Gamma( G_k T ), \F_p ) \ll p^{(1-\delta)k},
\end{equation*}

from which the proposition follows by conjugating by a power of $g_r$.  The base case of this is equation (\ref{roundtower}) (or its analogue for $\widehat{H}^1$), and for the inductive step we will use the following inclusion-exclusion inequality

\begin{multline*}
\dim \widehat{H}^1(\Gamma( G_k A_{i+1,k} ), \F_p ) + \dim \widehat{H}^1(\Gamma( G_k A'_{i+1,k} ), \F_p ) \\
- \dim \widehat{H}^1(\Gamma( G_k A_{i+1,k} A'_{i+1,k} ), \F_p ) \le \dim \widehat{H}^1(\Gamma( G_k ( A_{i+1,k} \cap A'_{i+1,k} ), \F_p ).
\end{multline*}

This is simply expressing the fact that both of the first two cohomology spaces is contained in the fourth, and their intersection is the third.  We may simplify this using the fact that $A_{i,k}$ and $A_{i,k}'$ are conjugate in $G$, so $\dim \widehat{H}^1(\Gamma( G_k A_{i+1,k} ), \F_p ) = \dim \widehat{H}^1(\Gamma( G_k A'_{i+1,k} ), \F_p )$.  We may also apply the inclusion $A_{i-1,k} \subseteq A_{i+1,k} \cap A'_{i+1,k}$ (valid because we chose $l \ge m$) and the equality $A_{i,k}A_{i,k}' = A_{i,k}N'_k$ to obtain

\begin{multline*}
2\dim \widehat{H}^1(\Gamma( G_k A_{i+1,k} ), \F_p ) \le \dim \widehat{H}^1(\Gamma( G_k A_{i,k} ), \F_p ) + \dim \widehat{H}^1(\Gamma( G_k A_{i,k} N'_k ), \F_p ).
\end{multline*}

Our inductive hypothesis provides a bound on the first term of the right hand side, and to bound the second term we will use the discussion above about the commensurator of $\Gamma$ which will allow us to `round out' the congruence subgroup $G_k A_{i,k} N'_k$ of $G$ and apply the inductive hypothesis again.  As conjugation by $g_r$ leaves the subgroups $A_{i,k}$ invariant, expands the upper triangular subgroups $N'_k$ by a factor of $p^r$ and contracts the lower triangular subgroups by $p^{-r}$, we may take $l = m = 2r$ to obtain the equality $G_{k-r}A_{i,k} = g_r^{-1} G_k A_{i,k} N'_k g_r$.  Therefore 

\begin{eqnarray*}
\dim \widehat{H}^1(\Gamma( G_k A_{i,k} N'_k ), \F_p ) & = & \dim \widehat{H}^1(\Gamma( G_{k-r} A_{i,k} ), \F_p )\\
 & \le & \dim \widehat{H}^1(\Gamma( G_{k-r} A_{i-1,k-r} ), \F_p )\\
 & \le & C \alpha^{i-1} p^{2(k-r)},
\end{eqnarray*}

where the last inequality is our inductive hypothesis.  Combining the two inequalities we have

\begin{eqnarray*}
2\dim \widehat{H}^1(\Gamma( G_k A_{i+1,k} ), \F_p ) & \le & C \alpha^i p^{2k} + C \alpha^{i-1} p^{2(k-r)} \\
\dim \widehat{H}^1(\Gamma( G_k A_{i+1,k} ), \F_p ) & \le & (\alpha^{-1}/2 + \alpha^{-2} p^{-2r}/2 ) C \alpha^{i+1} p^{2k},
\end{eqnarray*}

and it is clear that we may choose $\alpha$ so that the term in brackets is at most $1$.  This establishes (\ref{induct}) for all $i$, and as a consequence

\begin{equation}
\label{injlim}
\dim \widehat{H}^1( \Gamma( G_k T ), \F_p ) \ll p^{(2-\delta)k}.
\end{equation}

We have the same bound for $H^1( \Gamma( G_k T ), \F_p )$ by inequality (\ref{complete}).  The proposition follows by conjugating by a suitable power of $g_r$ and recalling our freedom to pass to normal subgroups of $p$-power index.

\begin{remark}
If the prime $2$ is totally split in $F$ and we may take $r=1$ in the above (this will be the case for the Bianchi group $\SL(2,\OO_7)$, for instance) then we may choose $\delta$ in the theorem to be $(2 - \ln_2( 1 + \sqrt{3} ) )/4 > 1/8$.
\end{remark}

\section{Examples}
\label{coh4}

In this section we verify the hypothesis of theorem \ref{weight} for the case $\Gamma = \SL(2,\OO_{-2})$ and $\p, \overline{\p}$ the two conjugate primes above 3.  We also give examples of towers of rational homology 3-spheres using proposition \ref{analytic}.  For small $p$ this is done by direct computation, and for large $p$ probable examples are produced by computing the dimensions of $H^1(\Gamma, \Sym^d \circ \pi_\p)$ and assuming the nondegeneration of a term in the Leray spectral sequence.  Here $\Sym^d$ are the symmetric power representations of $\SL(2,\F_p)$ with $\F_p$ coefficients and $\pi_\p$ is the quotient map $\Gamma \rightarrow \Gamma / \Gamma(\p) \simeq \text{SL}(2,\F_p)$.

\subsection{Verification of Theorem \ref{weight} }
\label{coh41}

As mentioned, we shall consider $\Gamma = \SL(2,\OO_{-2})$ and choose our two primes to be $\p$ and $\overline{\p}$, where $\p \overline{\p} = 3$.  Note that this choice is not arbitrary, but rather was suggested by \cite{CD} where they prove that $H^1_{\text{cusp}}( \Gamma(\p^k \overline{\p} ), \Q ) = 0$ for all $k$ assuming standard conjectures from number theory.  It turns out to be infeasible to compute $H^1( \Gamma(\p^k \overline{\p} ), \F_3 )$ for large enough $k$ to show that the bound of proposition \ref{saving} is satisfied, so instead we shall directly verify the consequences of that bound which were used in the proof.  Our first step in bounding $H^1( \Gamma(\p^k \overline{\p} ), \F_3 )$ will be to apply a decomposition of the regular representation $\F_3[ SL(2,\F_3) ]$ from \cite{BN} as follows:

\begin{eqnarray*}
\dim H^1( \Gamma(\p^k \overline{\p} ), \F_3 ) & = & \dim H^1( \Gamma, \F_3[ \Gamma / \Gamma(\p^k \overline{\p} ) ] ) \\
& = & \dim H^1( \Gamma, \F_3[ \Gamma / \Gamma(\p^k) ] \otimes \F_3[ \Gamma / \Gamma(\overline{\p}) ] ) \\
& \ll & \sum_{d=0}^2 \dim H^1( \Gamma, \F_3[ \Gamma / \Gamma(\p^k) ] \otimes \Sym^d ).
\end{eqnarray*}

Here $\Sym^d$ is the symmetric power representation of $\SL(2,\F_3)$ composed with reduction at $\overline{\p}$.  As in section 2 these cohomology groups may be approximated by the kernels of two maps $\partial \oplus \delta$ and $\Delta$, acting on modules which are direct sums of $\F_3[ G/G(3^k) ] \otimes \Sym^d$, and our goal is to show that the dimension of both kernels is $\ll 3^{2k}$.  We describe this for $\Delta$, the other case being identical.  $\Delta$ is a map from $\F_3[ G/G(3^k) ] \otimes \Sym^d$ to itself given by an element of $\Z[\Gamma]$, and we will show that $\ker \Delta$ is small by showing that $\im \Delta^t = \Delta$ is large.

Let $G = SL(2,\Z_3)$ as before, thought of as the completion of $\Gamma$ at $\p$.  Define $\F_3 \llbracket G( 3 ) \rrbracket^{\le n}$ to be the quotient of this completed group ring by the ideal of elements of degree $> n$, and $\F_3 \llbracket G \rrbracket^{\le n}$ be the induction of this from $G(3)$ to $G$.  We define representations of $\Gamma$ as follows:

\begin{eqnarray*}
V & = & \F_3 \llbracket G \rrbracket \otimes \Sym^d, \\
V_k & = & \F_3[ G / G(3^k) ] \otimes \Sym^d,\\
V^{\le n} & = & \F_3 \llbracket G \rrbracket^{\le n} \otimes \Sym^d.
\end{eqnarray*}

$V$ is a free $\F_3 \llbracket G(3) \rrbracket$ module, and we may choose a free basis $\{ v_{j,k} \}$ for it by letting $\{ \gamma_j \}$ be a set of coset representatives for $G/G(3)$, $\{ u_i \}$ be a basis for $\Sym^d$, and setting $v_{j,k} = \gamma_j \otimes u_k$.  Because $\Delta$ commutes with the right $\F_3 \llbracket G(3) \rrbracket$ action, to bound $H^1(\Gamma, V_k)$ it suffices to show that for each $(j,k)$ there is an element of $\im \Delta$ whose unique lowest degree co-ordinate with respect to the the basis $\{ v_{j,k} \}$ occurs at $v_{j,k}$.  This may be checked by computing the image of $\Delta$ between the modules $V^{\le n}$ and showing it contains terms supported in the $(j,k)$th co-ordinate for all $(j,k)$.  We have done this in the case of both operators for $n = 4$, thus establishing a power saving in the dimensions of their kernels as required.  Checking the constants produced by theorem \ref{weight} in this case and combining with the lower bound of \cite{FGT} establishes the following:

\begin{cor}
$d^{15/8} \gg \dim H^1( \SL(2,\OO_{-2}), E_d ) \gg d $.
\end{cor}

\subsection{Rational Homology Spheres}
\label{coh42}

We now summarise our computations testing the hypotheses of propositon \ref{analytic}.  As we have discussed, towers of rational homology spheres of the type which the proposition exhibits have already been seen in \cite{BE} and \cite{CD}, and the purpose of this section is merely to demonstrate that they are relatively common.   We have searched for examples among a family of orbifolds depending on two integer parameters called the twist knot orbifolds, denoted $T(n,k)$, which are described in more detail in  \cite{CD} and \cite{HS}.  The corresponding lattices $\Gamma$ are naturally subgroups of $PSL(2,\C)$, but passing to their $SL(2,\C)$ double cover does not alter their $\F_p$ cohomology for $p \neq 2$, and so we freely identify the two.  Establishing the hypothesis of proposition \ref{analytic} involves finding a degree 1 prime $\p$ of the trace field of $\Gamma$ at which $\Gamma$ unramified and showing that $\dim H^1( \Gamma(\p), \F_p) = 3$.  If $N(\p) \le 50$ or so this may be checked directly, while for large $\p$ we shall express this cohomology group as $H^1( \Gamma, \F_p [ \PSL(2,\F_p) ] )$ (as we may neglect the centre) and use the decomposition of this regular representation given in \cite{BN}.  This decomposition may be stated as

\begin{equation*}
\F_p [ \PSL(2,\F_p) ] = (\Sym^{p-1})^{\oplus p} \oplus W \oplus \bigoplus_{ \substack{  2 \le i \le p-3, \\ i \text{ even }  } }  V_i^{\oplus i+1},
\end{equation*}

where the decomposition series of $W$ and $V_i$ are

\begin{equation*}
\{ 1, \Sym^{p-3}, 1 \} \quad \text{and} \quad \{ \Sym^i, \Sym^{p-i-1} \oplus \Sym^{p-i-3}, \Sym^i \}
\end{equation*}

respectively.  We have

\begin{equation*}
\dim H^1( \PSL(2, \F_p), \Sym^{p-3} ) = 1,
\end{equation*}

and it appears that the pullback of this class to $\Gamma$ is always nonzero so that the smallest dimensions the cohomology groups $H^1( \Gamma, \Sym^d )$ can have is

\begin{equation}
\label{small}
\begin{array}{rl}
\dim H^1( \Gamma, \Sym^d ) = 1 & \text{for } d = p-3, \\
 = 0  & \text{otherwise}.
\end{array}
\end{equation}

If (\ref{small}) holds, we may use the Leray spectral sequence which computes $H^1(\Gamma(\p), \F_p)$ from $H^1( \Gamma, \Sym^d )$ to see that the expected dimension of this group is 3, as required.  $\Sym^{p-3}$ occurs as a composition factor in $V_2$, which has multiplicity 3 in $\F_p [ \PSL(2,\F_p) ]$, and we expect this to be its only contribution to $H^1(\Gamma(\p), \F_p)$.  When it appears in $W$ and as a subrepresentation of $V_{p-3}$, its contribution should be cancelled by $H^0( \Gamma, \F_p)$, which can be easily checked by constructing the relevant extension.  The contribution coming from its appearance as a quotient of $V_{p-3}$ should be cancelled by $H^2( \Gamma, \Sym^{p-3})$, which will be nonzero by Poincare duality.  We have been unable to check this for large $\p$, but it still seems likely that $\dim H^1(\Gamma(\p), \F_p)$ will be 3 if we are in the situation of (\ref{small}).

The table below summarises our experimental data, listing the parameters of the various orbifolds considered, whether they are arithmetic or not, the number of primes $\p$ tested and the number for which the dimensions of $H^1$ were as in (\ref{small}).\\

\begin{tabular}{r | ccccccc}
$(n,k)$ & $(-3,3)$ & $(-2,3)$ & $(2,3)$ & $(3,3)$ & $(4,3)$ & $(3,4)$ & $(4,4)$ \\
\hline
Arithmetic & y & y & y & y & n & n & n \\
No. primes & 83 & 80 & 79 & 35 & 82 & 38 & 36\\
No. analytic & 8 & 6 & 2 & 6 & 9 & 6 & 5\\
\end{tabular}

\end{document}